\documentclass[11pt,a4paper]{amsart}
\usepackage[latin1]{inputenc}
\usepackage[T1]{fontenc}
\usepackage{textcomp}
\usepackage{lmodern}

\renewcommand{\Re}{\operatorname{Re}}

\usepackage{hyperref}
\usepackage{xcolor}
\hypersetup{%
    colorlinks=false,
    pdfborderstyle={/S/U/W 1}
}

\usepackage{xspace}
\usepackage{enumitem}
\setenumerate[1]{label=(\roman{*}), ref=(\roman{*})}
\setenumerate[2]{label=(\alph{*}), ref=(\alph{*})}

\theoremstyle{plain}
\newtheorem{theorem}{Theorem}[section]
\newtheorem{lemma}[theorem]{Lemma}
\theoremstyle{definition}
\newtheorem{definition}[theorem]{Definition}
\newtheorem{example}[theorem]{Example}
\newtheorem*{ack*}{Acknowledgments}
\theoremstyle{remark}
\newtheorem{remark}[theorem]{Remark}

\DeclareMathOperator{\linspan}{span}
\DeclareMathOperator{\diam}{diam}

\begin{document}
\title{Relatively weakly open convex combinations of slices}

\author[T.~A.~Abrahamsen]{Trond A.~Abrahamsen}
\address[T.~A.~Abrahamsen]{Department of Mathematics, University of
  Agder, Postboks 422, 4604 Kristiansand, Norway.}
\email{trond.a.abrahamsen@uia.no}
\urladdr{http://home.uia.no/trondaa/index.php3}

\author[V.~Lima]{Vegard Lima}
\address[V.~Lima] {NTNU, Norwegian University of Science and
  Technology, Aalesund, Postboks 1517, N-6025 {\AA}lesund Norway.}
\email{Vegard.Lima@ntnu.no}

\date{\today}


\begin{abstract}
  We show that $c_0$, and in fact $C(K)$ for any
  scattered compact Hausdorff space $K$, have the
  property that finite convex combinations of slices
  of the unit ball are relatively weakly open.
\end{abstract}

\maketitle

\section{Introduction}
\label{sec:introduction}
Let $X$ be a (real or complex) Banach space with unit ball
$B_X$, unit sphere $S_X$, and dual $X^*$.
Given $x^* \in S_{X^*}$ and $\varepsilon > 0$
we define a slice of $B_X$ by
\begin{equation*}
  S(x^*,\varepsilon) :=
  \{x \in B_X : \Re x^*(x) > 1 - \varepsilon
  \},
\end{equation*}
where $\Re x^*(x)$ denotes the real part of $x^*(x)$.

Recall the following successively stronger ``big-slice concepts'',
defined in \cite{ALN01}:

\begin{definition} \label{defn:diam2p} A Banach space $X$ has the
  \begin{enumerate}
  \item
    \emph{local diameter 2 property} if every slice of $B_X$ has
    diameter 2.
  \item
    \emph{diameter 2 property} if every
    non-empty relatively weakly open subset of $B_X$ has diameter 2.
  \item
    \emph{strong diameter 2 property} if every finite convex
    combination of slices of $B_X$ has diameter 2.
  \end{enumerate}
\end{definition}
Since a slice is relatively weakly open,
the diameter 2 property implies the local diameter 2 property.
For some spaces $X$ the converse is also true.
For example, it is known that if every $x \in S_X$ is an extreme
point of $B_{X^{**}}$, then every non-empty relatively weakly
open subset of $B_X$ contains a slice by Choquet's lemma
(cf. e.g. Proposition~1.3 in \cite{AHNTT}).

By Bourgain's lemma \cite[Lemma~II.1]{MR912637}
every non-empty relatively weakly open subset of $B_X$ contains
a finite convex combination of slices hence the strong diameter 2 property
implies the diameter 2 property.
On a particularly sunny day at a conference at the University
of Warwick in 2015, Olav Nygaard asked if the converse is ever true.
The aim of this short note is to answer this question
affirmatively by showing that both $c_0$, and in fact $C(K)$
for any scattered compact Hausdorff space $K$, have the
much stronger property that finite convex combinations of slices
of the unit ball are relatively weakly open.
See Theorems~\ref{thm:coco-slices-for-CK} and
\ref{thm:coco-slices-for-c0} below.

Let us note that in general it is not true that
finite convex combinations of slices of the unit ball
are relatively weakly open.
The Banach space $\ell_2$ is one example \cite[Remark~IV.5]{MR912637}.
In their proof that the strong diameter 2 property is stronger
than the diameter 2 property Haller, Langemets and P{\~o}ldvere
\cite{HL1} show that if $Z$ is an $\ell_p$-sum of two Banach spaces,
$Z = X \oplus_p Y$ with $1 < p < \infty$, then
for every $\lambda \in (0,1)$ there exists two slices $S_1$ and $S_2$,
and a $\beta > 0$ such that
$\lambda S_1 + (1-\lambda) S_2 \subset (1-\beta)B_Z$.

We should also remark that the positive part of the unit
sphere of $L_1[0,1]$,
$F = \{f \in L_1[0,1] : f \ge 0, \|f\|=1\}$
is another example of a closed convex bounded subset
of a Banach space that satisfies a converse
to Bourgain's lemma  in that finite convex combinations
of slices of $F$ are relatively weakly open
\cite[Remark~IV.5]{MR912637}.

The notation and conventions we use are standard and follow
e.g. \cite{MR2766381}.

\section{Main result}
\label{sec:main-result}

We start by recalling the following
definition (see e.g. \cite[Definition~14.19]{MR2766381}).
\begin{definition}
  A compact space $K$ is said to be \emph{scattered compact}
  if every closed subset $L \subset K$ has an isolated point in $L$.
\end{definition}

Let $K$ be scattered compact and
consider the Banach space $C(K)$
of all continuous functions on $K$ with $\sup$-norm.
Rudin \cite{MR0085475} showed that $C(K)^{*} = \ell_1(K)$
in this case. Pe{\l}czy{\'n}ski and Semadeni \cite{MR0107806} showed
that for a compact Hausdorff space $K$ we have
$C(K)^* = \ell_1(K)$ if and only if $K$ is scattered (= dispersed).

To prove the main result, we will need the following geometric lemma
for the unit circle in the complex plane.

\begin{lemma}\label{lem:pointsoncircle}
  Let $e^{i \alpha}$ and $e^{i \beta}$ be distinct
  points on the unit circle with distance
  $d = |e^{i \alpha} - e^{i\beta}|$.
  If $0 < \mu < \frac{1}{2}$,
  then the point
  $c = \mu e^{i \alpha} + (1 - \mu)e^{i \beta}$
  on the line segment between $e^{i\alpha}$ and $e^{i \beta}$
  satisfies
  \begin{equation*}
    |c| \le 1 - \frac{d^2 \mu }{4}.
  \end{equation*}
\end{lemma}

\begin{proof}
  A straightforward calculation shows that
  $d^2 = 2 - 2\cos(\alpha - \beta)$
  and that $|c|^2 = \mu ^2 + (1-\mu )^2 + \mu (1-\mu )2\cos(\alpha-\beta)$.
  Hence $|c|^2 = 1 - d^2 \mu(1 - \mu)$.
  Since $\sqrt{1+x} \le 1 + \frac{x}{2}$ for $x \ge -1$
  and $\mu (1-\mu ) \ge \frac{\mu }{2}$ for $\mu  \in [0,\frac{1}{2}]$ we get
  \begin{equation*}
    |c| = \sqrt{1 - d^2\mu (1-\mu )}
    \le 1 - \frac{1}{2}d^2 \mu (1-\mu )
    \le 1 - \frac{d^2\mu }{4}
  \end{equation*}
  as desired.
\end{proof}

\begin{theorem}\label{thm:coco-slices-for-CK}
  Let $K$ be scattered compact. Then every finite convex combination of slices of
  the unit ball of $C(K)$ is relatively weakly open.
\end{theorem}

\begin{proof}
  Let $\{S(f_i,\varepsilon_i)\}_{i=1}^k$
  be slices of $B_{C(K)}$ and let $\lambda_i > 0$
  with $\sum_{i=1}^k \lambda_i = 1$, and consider the
  convex combination of these slices
  \begin{equation*}
    C = \sum_{i=1}^k \lambda_i S(f_i,\varepsilon_i).
  \end{equation*}
  Let $x = \sum_{i=1}^k \lambda_i z_i \in C$
  with $z_i \in S(f_i,\varepsilon_i)$.
  Our goal is to find a non-empty relatively weakly open
  neighborhood of $x$ that is contained in $C$.

  Let $d = \min\{\Re f_i(z_i) - (1-\varepsilon_i) : 1 \le i \le k\}$
  and let $\eta > 0$ be such that $\eta < d/3$.
  Let $E \subset K$ be a finite set such that
  $\sum_{t \notin E} |f_i(t)| < \eta$
  for $1 \le i \le k$.

  Define
  \begin{equation*}
    \mathcal{U} =
    \left\{
      y \in B_{C(K)} : |y(t) - x(t)| < \delta,
      t \in E
    \right\}
  \end{equation*}
  where $\delta > 0$. Next we specify how $\delta$ is chosen.

  Let $L = \max\{\frac{1}{\lambda_i} : i=1,2,\ldots,k \}$.
  Let
  \begin{equation*}
    E_I =
    \left\{
      t \in E : \text{there exists} \;
      1 \le i_0 \le k \; \text{such that}\; |z_{i_0}(t)| < 1
    \right\}.
  \end{equation*}
  Define $\delta_I = (1+3L)^{-1} \min_{t \in E_I} (1-|z_{i_0}(t)|)$.
  Let
  \begin{equation*}
    E_{III} =
    \left\{
      t \in E \setminus E_I :
      \text{there exists} \; i \neq j
      \; \text{such that}\;
      z_i(t) \neq z_j(t)
    \right\}
  \end{equation*}
  and define
  \begin{equation*}
    D =
    \min_{t \in E_I} \min_{z_i(t) \neq z_j(t)} \{|z_i(t)-z_j(t)|^2\}.
  \end{equation*}
  Choose $0 < \rho < \min \{ D/8, \eta/4L \}$.
  Define $\delta_{III} = D\rho (4(1+3L))^{-1}$.
  Finally we choose $\delta < \min\{\eta/6L, \delta_I, \delta_{III}\}$.

  Let $y \in \mathcal{U}$.
  Define $w(t) = y(t) - x(t)$ for all $t \in K$.
  We will define $\bar{z}_i \in S(f_i,\varepsilon_i)$,
  $i=1,2,\ldots,k$, and show that $y$ can be written
  $y = \sum_{i=1}^k \lambda_i \bar{z}_i \in C$.

  If $t \in E$ is an isolated point let $\mathcal{V}_t = \{t\}$.
  While for a non-isolated point $t \in E$ we let $\mathcal{V}_t$
  be a neighborhood of $t$ chosen so that $\mathcal{V}_t
  \cap \mathcal{V}_{t'} = \emptyset$ for all $t' \neq t$, $t' \in E$,
  and $|z_i(t) - z_i(s)| < \delta$, $1 \le i \le k$,
  $|x(t)-x(s)| < \delta$
  and $|y(t)-y(s)| < \delta$
  for all $s \in \mathcal{V}_t$.
  In particular, we get $|x(s)-y(s)| < 3\delta$
  for all $s \in V_t$.

  \noindent
  \textbf{Definition of $\bar{z}_i$ outside
    $\bigcup_{t \in E} \mathcal{V}_t$.}

  For $s \in K \setminus \cup_{t \in E} \mathcal{V}_t$
  we define $\bar{z}_i(s) = y(s)$
  for all $1 \le i \le k$.

  \noindent
  \textbf{Definition of $\bar{z}_i$ on
    $\bigcup_{t \in E} \mathcal{V}_t$.}

  Let $t \in E$.

  \noindent
  \emph{Case I:}
  If there exists $i_0$ with $|z_{i_0}(t)| < 1$,
  then we choose by Urysohn's lemma
  a real-valued non-negative continuous function  $n_t \in S_{C(K)}$ with
  $n_t(t) = 1$ such that $n_t(s) = 0$ off
  $\mathcal{V}_t$. Now, for $s \in \mathcal{V}_t$
  let
  \begin{equation*}
    \bar{z}_{i_0}(s) = n_t(s)[z_{i_0}(s) +
    \lambda_{i_0}^{-1} w(s)] + [1-n_t(s)]y(s)
  \end{equation*}
  and for $i \neq i_0$ we let
  \begin{equation*}
    \bar{z}_i(s) = n_t(s)z_i(s) +
    [1-n_t(s)]y(s).
  \end{equation*}
  Then
  \begin{align*}
    \sum_{i=1}^k \lambda_i \bar{z_i}(s)
    &= n_t(s)w(s) + \sum_{i=1}^k \lambda_i
      \bigg[n_t(s)z_i(s) + (1- n_t(s))y(s)\bigg]\\
    &= n_t(s)(y(s) - x(s)) + n_t(s)
      x(s) + (1-n_t(s))y(s)= y(s).
  \end{align*}
  By the choice of $\delta$
  \begin{align*}
    |z_{i_0}(s) + \lambda_{i_0}^{-1} w(s)|
    &\le
      |z_{i_0}(t)| + |z_{i_0}(s) - z_{i_0}(t)|
      + L |y(s) - x(s)| \\
    &\le
      |z_{i_0}(t)| + \delta + 3L\delta
      < 1,
  \end{align*}
  thus we have $|\bar{z}_{i}(s)| \le 1$ for every $1 \le i \le k$.
  Moreover, for each $1 \le i \le k$ the function $\bar{z}_{i}$
  is continuous on $K$ since $z_{i}$, $y$, $x$ and $n_t$ all are and since
  $n_t$ is zero off $\mathcal{V}_t$ so $\bar{z}_{i} = y$ there.

  We will need that
  $|\bar{z}_{i_0}(t) - z_{i_0}(t)|
  \le \lambda_{i_0}^{-1}|y(t)-x(t)| < L\delta < \eta$
  and $|\bar{z}_{i}(t) - z_{i}(t)| = 0$ for $i \neq i_0$.

  \noindent
  \emph{Case II:}
  If for all $1 \le i,l \le k$ we have $z_i(t) = z_l(t)$ with $|z_i(t)|= 1$,
  then $x(t) = z_i(t)$ and we can just let
  $\bar{z}_i(s) = y(s)$ for all $1 \le i \le k$
  and $s \in \mathcal{V}_t$. Continuity of $\bar{z}_i$ on $K$
  in this case is trivial.

  We will need that $|\bar{z}_i(t)-z_i(t)|
  = |y(t)-x(t)| < \delta < \eta$.

  \noindent
  \emph{Case III:}
  If $|z_i(t)| = 1$ for all $1 \le i \le k$,
  but not all $z_i(t)$ are equal.
  Order the set $\{ \arg z_i(t) :  1 \le i \le k\}$
  as an increasing sequence
  $\{\theta_1 < \theta_2 < \cdots < \theta_q\}$
  and define $\theta_0 = \theta_q$.
  We put $A_p = \{i: \arg z_i(t) =
  \theta_p\}$ and $\Lambda_p = \sum_{i \in A_p} \lambda_i$.

  With $\rho$ as above we define for $1 \le p \le q$
  \begin{equation*}
    c_p = \rho(e^{i\theta_{p-1}} - e^{i\theta_p}).
  \end{equation*}
  Let $s \in \mathcal{V}_t$ and define (for $i \in A_p$)
  \begin{equation*}
    \bar{z}_i(s) = n_t(s)\bigg[z_i(s) +
    \frac{c_p}{\Lambda_p} + \frac{w(s)}{q\Lambda_p}\bigg] + (1-n_t(s))y(s).
  \end{equation*}

  We have
  \begin{multline*}
    \sum_{i=1}^k \lambda_i \bar{z}_i(s)
    = \sum^q_{p =1}\sum_{i \in A_p}\lambda_i \bar{z}_i(s)\\
    =
    \sum^q_{p=1}
    n_t(s)
    \sum_{i \in A_p} \lambda_i z_i(s) + \sum^q_{p=1} n_t(s)c_p +
    \sum^q_{p=1}n_t(s)\frac{w(s)}{q}
    +  (1-n_t(s))y(s)
    \\
    =n_t(s)\sum_{i=1}^k \lambda_i z_i(s) +
    n_t(s)0 + n_t(s)w(s) + (1-n_t(s))y(s)\\
    = n_t(s) x(s) + n_t(s)(y(s) -
    x(s)) + y(s) - n_t(s) y(s)
    = y(s).
  \end{multline*}

  With $\mu = \rho/\Lambda_p$
  \begin{equation*}
    z_i(t) + \frac{c_p}{\Lambda_p}
    =
    e^{i \theta_p} + \mu(e^{i\theta_{p-1}} - e^{i\theta_p})
    =
    \mu e^{i\theta_{p-1}} + (1-\mu)e^{i\theta_p}.
  \end{equation*}
  so by Lemma~\ref{lem:pointsoncircle}
  \begin{equation*}
    |z_i(t) + \frac{c_p}{\Lambda_p}|
    \le 1 - \frac{|e^{i\theta_{p-1}}-e^{i \theta_p}|^2
      \rho}{4\Lambda_p}
    \le 1 - \frac{D\rho}{4\Lambda_p}
    < 1 - \frac{D\rho}{4}
    < 1 - (1+3L)\delta
  \end{equation*}
  hence
  \begin{align*}
    \left|
    z_i(s) + \frac{c_p}{\Lambda_p} + \frac{w(s)}{q\Lambda_p}
    \right|
    &\le
    |z_i(t) + \frac{c_p}{\Lambda_p}|
    +
    |z_i(s) - z_i(t)|
    +
    \left|\frac{w(s)}{q\Lambda_p}\right| \\
    &< 1 - (1+3L)\delta
    + \delta + 3L\delta = 1.
  \end{align*}
  Thus we have $|\bar{z}_i(s)| \le 1$.
  We will also need that
  \begin{equation*}
    |\bar{z}_i(t)-z_i(t)|
    = \left|
      \frac{c_p}{\Lambda_p}
      + \frac{w(t)}{q\Lambda_p}
    \right|
    \le \rho|e^{i\theta_{p-1}} - e^{i\theta_p}|L
    + 3\delta L
    \le 2L \rho + 3L \delta
    \le \eta.
  \end{equation*}

  Also note that for each $1 \le i \le k$ the function $\bar{z}_{i}$
  is continuous on $K$ since $z_{i}$, $y$, $x$ and $n_t$ all are and since
  $n_t$ is zero off $\mathcal{V}_t$ so $\bar{z}_{i} = y$ there.

  \noindent
  \textbf{Conclusion.}

  So far we have defined $\bar{z}_i \in B_{C(K)}$ and shown
  that $y = \sum_{i=1}^k \lambda_i \bar{z}_i$.
  It only remains to show that $\bar{z}_i \in S(f_i,\varepsilon_i)$.
  We have
  \begin{equation*}
    \sum_{t \notin E} |f_i(t)
    ( \bar{z}_i(t) - z_i(t) )| < \eta\|\bar{z}_i - z_i\|
    \le 2 \eta,
  \end{equation*}
  and
  \begin{equation*}
    \sum_{t \in E} |f_i(t)
    ( \bar{z}_i(t) - z_i(t) )|
    < \|f_i\|\eta
    < \eta.
  \end{equation*}
  Hence $|f_i(\bar{z}_i - z_i)| < 3\eta$
  so that
  \begin{equation*}
   \Re f_i(\bar{z}_i) \ge \Re f_i(z_i) - 3\eta
    > \Re f_i(z_i) - d > 1 - \varepsilon_i,
  \end{equation*}
  and we are done.
\end{proof}

The above theorem applies to $C[0,\alpha]$
for any infinite ordinal $\alpha$, and
in particular to $c = C[0,\omega]$.
It should be clear that the proof also works
for real scalars and that
it proves the following result.

\begin{theorem}\label{thm:coco-slices-for-c0}
  Every finite convex combination of slices of
  the unit ball of $c_0$ is relatively weakly open.
\end{theorem}

\section{Questions and remarks}
\label{sec:questions-remarks}

We will end with some questions and remarks.

\begin{enumerate}
\item\label{item:q1}
  Which Banach spaces satisfy that
  finite convex combinations of slices of the unit ball
  are relatively weakly open?
  Does this hold for spaces with the
  strong diameter 2 property?
\item\label{item:q2}
  Which Banach spaces satisfy that
  finite convex combinations of slices of the unit ball
  contain a non-empty relatively weakly open neighborhood
  of some point in the combination?
\item\label{item:q3}
  Which Banach spaces satisfy that
  finite convex combinations of slices of the unit ball
  always have non-empty intersection with the sphere?
\item\label{item:q4}
  If finite convex combinations of slices
  of both $B_X$ and $B_Y$ are relatively weakly
  open, is the same true for the unit ball
  of $X \oplus_\infty Y$ and/or $X \oplus_1 Y$?
\end{enumerate}

It is not clear that there is a connection
between having relatively weakly open
convex combinations of slices
and the diameter two properties.
But we have the following observation.

\begin{remark}
  Let $X$ be a Banach space such that
  there exists a slice $S_1 = S(x^*,\varepsilon)$
  of $B_X$ with $\diam S_1 < 1$, then there is
  a convex combination of slices of $B_X$ that
  is not relatively weakly open.

  Define $S_2 = S(-x^*,\varepsilon)$. Since $x \in S_1$ if and only
  if $-x \in S_2$ we have that $S_2$ satisfies
  $\diam S_2 = \diam S_1 < 1$.

  Let $C = \frac{1}{2}S_1 + \frac{1}{2}S_2$.  Fix $x \in S_1$. We
  have $\frac{1}{2}x + \frac{1}{2}(-x) = 0 \in C$.  If
  $\frac{1}{2}x_1 + \frac{1}{2}x_2 \in C$, then
  \begin{equation*}
    \|\frac{1}{2}x_1 + \frac{1}{2}x_2\|
    \le
    \frac{1}{2}\|x_1 - x\|
    +
    \frac{1}{2}\|x_2 - (-x)\|
    < 1
  \end{equation*}
  hence $C \cap S_X = \emptyset$ and $C$ is not relatively weakly
  open.
\end{remark}

Regarding Question~\ref{item:q3} we have
the following examples of spaces where
finite convex combinations of slices
intersect the sphere.

\begin{example}
  Finite convex combinations of slices of the unit ball
  of $L_1[0,1]$ always intersect the sphere.
  Here slices are given by functions $g_i \in B_{L_\infty[0,1]}$.
  We may assume that the $g_i$'s are simple functions
  and find sets $B_i \subset [0,1]$ with $B_i \cap B_j = \emptyset$
  and $\|\chi_{B_i} g_i\|_\infty$ almost one.
  The functions $f_i = m(B_i)^{-1} \chi_{B_i}$ does the job
  ($m$ is Lebesgue measure).
\end{example}

\begin{example}
  Let $X$  be a Banach space such that whenever
  $S_i = S(x_i^*,\varepsilon_i)$ with $x_i^* \in S_X$ and
  $\varepsilon_i > 0$ for $1 \le i \le n$,
  are $n$ slices of $B_X$, then there exists
  $x_i \in S_i \cap S_X$ and $y \in S_X$ such that
  $\|x_i \pm y\| = 1$ and $x_i + y \in S_i$.

  Spaces that satisfy this condition include
  $\ell_\infty^c(\Gamma)$ for $\Gamma$ uncountable
  since this space is almost square with $\varepsilon = 0$
  \cite[Remark~2.11]{ALL}.
  It also includes $\ell_\infty$ and $C[0,1]$
  since the slices there are defined by measures
  of bounded variation.

  If $X$ is a space with this property, then
  finite convex combinations of slices of $B_X$
  always intersect the sphere.

  Indeed, let $\lambda_i > 0$ with $\sum_{i=1}^n \lambda_i = 1$
  and let $S_i = S(x_i^*,\varepsilon_i)$
  be slices of $B_X$ with $x_i^* \in S_{X^*}$
  and $\varepsilon_i > 0$ for $1 \le i \le n$.

  By assumption, there exists
  $x_i \in S_i \cap S_X$ and $y \in S_X$ such that
  $\|x_i \pm y\| = 1$ and $x_i + y \in S_i$.

  Choose $y^* \in S_{X^*}$ such that $y^*(y) = 1$.
  Then
  \begin{equation*}
    1 = \|x_i \pm y\| \ge y^*(y) \pm y^*(x_i)
    = 1 \pm y^*(x_i)
  \end{equation*}
  hence $y^*(x_i) = 0$. Now $\sum_{i=1}^n \lambda_i (x_i + y) \in
  \sum_{i=1}^n \lambda_i S_i$ and
  \begin{equation*}
    \|\sum_{i=1}^n \lambda_i (x_i + y)\|
    \ge \sum_{i=1}^n \lambda_i y^*(y) = 1.
  \end{equation*}
\end{example}

\begin{example}
  If $X$ has the Daugavet property, then
  finite convex combinations of weak$^*$-slices
  of $B_{X^*}$ intersect the sphere $S_{X^*}$.
  To see this let $1 \le i \le n$,
  $x_i \in S_X$, $\varepsilon_i > 0$, and
  let $S(x_i,\varepsilon_i)$ be slices of $B_{X^*}$.

  Let $x_1^* \in S(x_1,\varepsilon_1) \cap S_{x^*}$
  and $V_1 = \linspan(x_1^*)$.
  By \cite[Lemma~2.12]{KadSSW}
  there exists $x_2^* \in S(x_2,\varepsilon_2) \cap S_{X^*}$
  such that $\|x_2^* + v^*\| = 1 + \|v^*\|$ for all $v^* \in V_1$.

  Let $V_k = \linspan(x_1^*,x_2^*,\ldots,x_k^*)$.
  Again by \cite[Lemma~2.12]{KadSSW}
  there exists
  $x_{k+1}^* \in S(x_{k+1},\varepsilon_{k+1}) \cap S_{X^*}$
  such that $\|x_{k+1}^* + v^*\| = 1 + \|v^*\|$ for all $v^* \in V_{k}$.

  If $\lambda_i > 0$ with $\sum_{i=1}^n \lambda_i = 1$,
  then it follows that
  $\sum_{i=1}^n \lambda_i x_i^*
  \in \sum_{i=1}^n \lambda_i S(x_i,\varepsilon_i)$ and
  \begin{equation*}
    \|\sum_{i=1}^n \lambda_i x_i^*\|
    = \lambda_n + \|\sum_{i=1}^{n-1} \lambda_i x_i^*\|
    = \cdots = \sum_{i=1}^n \lambda_i = 1.
  \end{equation*}
\end{example}

\begin{ack*}
  We thank Stanimir Troyanski and Olav Nygaard
  for fruitful conversations.
\end{ack*}

\providecommand{\MR}{\relax\ifhmode\unskip\space\fi MR }
\providecommand{\MRhref}[2]{%
  \href{http://www.ams.org/mathscinet-getitem?mr=#1}{#2}
}
\providecommand{\href}[2]{#2}

\end{document}